\documentclass{amsart}
\usepackage{latexsym,amsfonts,amssymb,amsthm,amsmath}
\newcommand{\tmop}[1]{\ensuremath{\operatorname{#1}}}

\newtheorem{theorem}{Theorem}[section]
\newtheorem{cor}[theorem]{Corollary}
\newtheorem{lemma}[theorem]{Lemma}
\newtheorem{prop}[theorem]{Proposition}
\theoremstyle{definition}

\newtheorem{conj}[theorem]{Conjecture}

\theoremstyle{plain}
\newtheorem*{theorem*}{Theorem}
\numberwithin{equation}{section}

\newcommand{\mc}[1]{\mathcal{#1}}

\newcommand{\ucp}[2]{\tmop{UCP}(#1, #2)}

\begin{document}
\begin{center}
\end{center}
\title{Korovkin-type properties for completely positive maps}
\author{Craig Kleski}
\address{Department of Mathematics, Miami University, Oxford, OH, USA}
\email{kleskic@miamioh.edu}
\begin{abstract}
Let $S$ be an operator system in $B(H)$ and let $A$ be its
generated $C^{\ast}$-algebra. We give a new characterization of
Arveson's unique extension property for unital completely positive
maps on $S$. We also show that when $A$ is a Type I
C*-algebra,  if every irreducible representation of $A$ is a
boundary representation for $S$, then every unital completely positive map on $A$ with codomain $A''$
that fixes $S$ also fixes $A$.

\thanks{
Thanks to David Sherman, Ken Davidson, and Matthew Kennedy for many
helpful conversations, comments, and suggestions on this topic.}

\end{abstract}
%\date{\today}
\maketitle
\begin{center}
\end{center}

\section{Introduction}

Korovkin-type properties for completely positive maps may be viewed as
an essential ingredient of what Arveson called ``noncommutative
approximation theory''. In \cite{Arveson08a}, he initiated this study
by investigating the rigidity of completely positive maps on
C*-algebras generated by operator systems.  We expand on this
study by characterizing the unique extension property for completely
positive maps with a Korovkin-type theorem.  Then we explore the
extent to which the noncommutative Choquet boundary determines this
rigidity for Type I C*-algebras for completely positive extensions having a
particular but natural codomain.  Finally, we discuss
Arveson's hyperrigidity conjecture and obtain structural information
about Type I C*-algebras generated by operator systems when
every irreducible representation is in the noncommutative Choquet boundary.

Let $X$ be a
compact Hausdorff metrizable topological space and let $C(X)$ be the continuous
complex-valued functions on $X$. Let $M$ be a \emph{function
  space} in $C(X)$; that is, a unital, linear subspace of $C(X)$ that separates
$X$.  Let $K(M)$ be the set $\{\phi\in
M^{\ast}:\phi(1)=1=\|\phi\|\}$. This is a compact convex subset of
$M^{\ast}$ called the \emph{states} of $M$.  Within $K(M)$ are the
evaluation functionals $\tmop{ev}_x$, and using them, we define the
Choquet boundary for $M$:
\begin{align*}
\partial_{M}:=\{x\in X:\tmop{ev}_x\text{ is an extreme point of }K(M)\}.
\end{align*} 
The space $M$ is called a \emph{Korovkin set} if whenever
$\{\phi_n\}$ is a sequence of positive maps from $C(X)$ to itself such
that $\|\phi_n(f)-f\|\to 0$ for all $f\in M$, then $\|\phi_n(f)-f\|\to 0$
for all $f\in C(X)$. The classical Korovkin theorem asserts that $\tmop{span}\{1,x,x^2\}$ is a Korovkin set for $C([0,1])$. 

In \cite{Saskin67}, {\v{S}}a{\v{s}}kin showed the following important connection between
certain Korovkin sets and Choquet boundaries:  a function space $M$ in
$C(X)$ is a Korovkin
set for $C(X)$ if and only if $\partial_M=X$. The
main result of this paper is to obtain a noncommutative version of
this result for a class of separable C*-algebras.

\begin{theorem*}
Let $S$ be a separable concrete operator system generating a Type I
C*-algebra $A$.   If every irreducible representation of $A$ is a boundary
representation for $S$, then for every unital completely positive map
$\psi:A\to A''$ satisfying $\psi(s)=s$ for all $s\in S$, we have
$\psi(a)=a$ for all $a\in A$.
\end{theorem*}

The noncommutative Choquet boundary has received a
considerable amount of attention lately.  In \cite{Arveson08},
Arveson proved that every separable operator system has ``sufficiently
many'' boundary representations.  This result was improved in
\cite{Kleski11}, where the author showed that the noncommutative
Choquet boundary is a boundary in the classical sense.  Recently,
Davidson and Kennedy (\cite{DK13}) generalized Arveson's result to the nonseparable
case, settling an influential problem from \cite{Arveson69}.  Despite its long history, many lines of inquiry in noncommutative Choquet theory remain unexplored.  The theory has
connections to other areas of operator algebras, including noncommutative convex sets and peaking phenomena for
operator systems. Korovkin-type properties for completely positive
maps are but a small part of a long list of potential applications.

We establish some definitions, notation, and conventions. Let $H$ be a complex Hilbert space and let $B(H)$ be the bounded
linear operators on $H$. A concrete \emph{operator system} is a unital
self-adjoint linear subspace of $B(H)$, and it serves as a
noncommutative generalization of a function space.  A \emph{completely positive} (cp)
map is a linear map $\phi$ between operator systems $S_1,S_2$ such
that for all $n$, $\phi^{(n)}:M_n(S_1)\to M_n(S_2)$, where
$\phi^{(n)}((s_{ij})):=(\phi(s_{ij}))$, is a positive map. When $\phi$
is unital, it is a \emph{unital completely positive} (ucp) map.  We denote by
$\ucp{S_1}{S_2}$ the set of ucp maps from $S_1$ to $S_2$.  A
\emph{representation} $\pi$ of a C*-algebra $A$ is a
$\ast$-homomorphism from $A$ to $B(K)$ for some Hilbert space $K$. All
representations will be assumed to be nondegenerate in the sense that the closed
span of  $\pi(A)K$ is all of $K$.

A C*-algebra $A$ is \emph{Type I} if $A^{\ast\ast}$ is a Type I von
Neumann algebra.  It is nontrivial that this is equivalent to being a
\emph{GCR} C*-algebra, meaning every irreducible representation
contains nonzero compact operators in its image. Among C*-algebras, the Type I algebras are
distinguished by having a nice representation theory, in
the sense that two irreducible representations are
unitarily equivalent if and only if they have the same kernel.  A wider class of
C*-algebras are those which are \emph{nuclear}; i.e., those
C*-algebras $A$ such that for any C*-algebra $B$,
there is a unique C*-cross norm on the algebraic tensor product of
$A$ and $B$ making it a C*-algebra.

Let $A$ be a C*-algebra and let $B$ be a 
 C*-subalgebra. A \emph{conditional expectation} from $A$ to
 $B$ is a completely positive projection of norm 1.  A
 C*-algebra $A$ is \emph{injective} if for every faithful
 representation $\pi$ of $A$ acting on the Hilbert space $K$, there
 exists a conditional expectation $E:B(K)\to \pi(A)$.  For example,
 when $A$ is a nuclear C*-subalgebra of $B(H)$, then $A''$ is injective. One consequence of this is
 that if $\psi:S\to A''$ is ucp, there is a ucp map
$\tilde{\psi}:A\to A''$ such that $\tilde{\psi}|_{S}=\psi$.

\section{Metric properties of  representations with the UEP}\label{sect2}

Let $S$ be a concrete operator system in $B(H)$ and let $A$ be the
(unital) 
C*-algebra that it generates. When $A$ and $H$ are separable, the spaces $\ucp{S}{B(H)}$ and $\ucp{A}{B(H)}$ are compact, Hausdorff, and metrizable in the \emph{bounded weak} or $BW$ topology.  In this topology, a sequence $\{\phi_n\}$ of ucp maps on (say) $S$ converges to a ucp map $\psi$ if for all $\xi,\eta\in H$ and all $s\in S$,
\begin{align*}
\langle (\phi_n(s)-\psi(s))\xi,\eta\rangle &\to 0,
\end{align*}
as $n\to \infty$. On $A$, when $\psi$ is a $\ast$-homomorphism, we can use this to get ``bounded strong-*'' (BS*) convergence.
\begin{lemma}\label{lemma1}
Let $\{\phi_n\}$ be a sequence of ucp maps from a  C*-algebra $A$ to $B(H)$ converging in the BW-topology on $\ucp{A}{B(H)}$ to a $\ast$-homomorphism $\pi$. Then for all $a\in A$ and all $\xi\in H$, 
\begin{align*}
\|(\pi(a)-\phi_n(a))\xi\| &\to 0.
\end{align*} 
\end{lemma}
\begin{proof}
The idea for the proof is from p. 57 of \cite{Davidson96}. 
\begin{align*}
\begin{split}
\|(\pi(a)-\phi_n(a))\xi\|^2 &= \langle(\pi(a)-\phi_n(a))\xi,\pi(a)\xi\rangle \\
&\quad + \langle (\pi(a^{\ast})-\phi_n(a^{\ast})\pi(a)\xi,\xi\rangle \\
&\quad \quad -\langle(\pi(a^{\ast}a)-\phi_n(a^{\ast})\phi_n(a))\xi,\xi\rangle.
\end{split}
\end{align*}
Note that $\phi_n(a^{\ast})\phi_n(a)\leq\phi_n(a^{\ast}a)$ for all $n$
by the Kadison-Schwarz inequality, and so 
\begin{align*}
\begin{split}
\|(\pi(a)-\phi_n(a))\xi\|^2 &\leq  \langle(\pi(a)-\phi_n(a))\xi,\pi(a)\xi\rangle \\
&\quad + \langle (\pi(a^{\ast})-\phi_n(a^{\ast}))\pi(a)\xi,\xi\rangle \\
&\quad \quad -\langle(\pi(a^{\ast}a)-\phi_n(a^{\ast}a))\xi,\xi\rangle,
\end{split}
\end{align*}
and each summand on the right goes to 0 as $n$ tends to $\infty$, for all $a\in A$ and all $\xi\in H$.
\end{proof}

Let $E\subset F$ be operator systems, and let
$\phi:E\to B(H)$ be a ucp map.  By a fundamental theorem in
\cite{Arveson69}, there exists a ucp map $\tilde{\phi}:F\to B(H)$ such
that $\tilde{\phi}|_E = \phi$; i.e., $\phi$ has an extension to a ucp
map from 
$F$ to $B(H)$. For an operator system $S\subseteq A:=C^{\ast}(S)\subseteq B(H)$,
if a representation $\pi$ of $A$ is such that $\pi|_{S}$ has a unique
extension to a ucp map from $A$ to $B(H)$, we say that $\pi$ has the \emph{unique
  extension property} (UEP) relative to $S$. (We also say ``$\pi|_{S}$
has the UEP''.)  In other words, $\pi|_{S}$ has the UEP if the only ucp
extension of $\pi|_{S}$ is $\pi$. An irreducible
representation with the UEP relative to $S$ is a \emph{boundary
  representation} for $S$.  Boundary representations are a
noncommutative analogue of Choquet points for a function space. In general, it is not so easy to determine when a representation has
the unique extension property, and few alternative characterizations
of the UEP 
are known.   In the next
proposition, we give a new characterization of the unique extension
property in terms of BW-convergence.

\begin{prop}\label{prop1}
Let $S$ be a separable concrete operator system, and let $A$ be the
C*-algebra it generates.  Let $\pi$ be a 
representation of  $A$ acting on a separable Hilbert space $H$. The following are equivalent:
\begin{enumerate}
\item $\pi|_{S}$ has the unique extension property;
\item for any sequence $\{\phi_n\}$ in $\ucp{A}{B(H)}$ such that $\{\phi_n|_{S}\}$ converges to $\pi|_{S}$ in the BW-topology on $\ucp{S}{B(H)}$, the sequence $\{\phi_n\}$ converges to $\pi$ in the BW-topology on\\ $\ucp{A}{B(H)}$;
\item for any sequence $\{\phi_n\}$ in $\ucp{A}{B(H)}$ such that $\{\phi_n|_{S}\}$ converges to $\pi|_{S}$ in the BW-topology on $\ucp{S}{B(H)}$, there exists a subsequence $\{\phi_{n_k}\}$ that converges to $\pi$ in the BW-topology on $\ucp{A}{B(H)}$.
\end{enumerate}
The same is true if we replace BW-convergence by BS-convergence.
\end{prop}

\begin{proof}
Note that (2) $\Rightarrow$ (3) is trivial. 

(1) $\Rightarrow$ (2): As noted above, the BW-topology on $\ucp{A}{B(H)}$ is Hausdorff and metrizable, and $\ucp{A}{B(H)}$ is also compact in this topology. Because of this, we can take advantage of elementary results about convergence of sequences in compact Hausdorff metrizable spaces.  The sequence $\{\phi_n\}$ must have a BW-convergent subsequence
$\{\phi_{n_k}\}$; call its limit $\phi$.  Now by assumption, $\{\phi_{n_k}|_{S}\}$ BW-converges to
$\pi|_{S}$. We see that $\phi|_{S}=\pi|_{S}$. Because $\pi$ is assumed
to have the UEP, we have $\phi=\pi$. Now every BW-convergent
subsequence of $\{\phi_n\}$ converges to $\pi$, which implies that $\{\phi_n\}$ itself 
BW-converges to $\pi$.

(3)$\Rightarrow$(1): Obvious. Let $\phi$ be an extension of $\pi|_{S}$; take $\phi_n=\phi$ for all $n$. Then (3) implies $\{\phi_n\}$ BW-converges to $\pi$ and so $\phi=\pi$.
\end{proof}

\section{Rigidity of ucp maps on nuclear C*-algebras}\label{sect3}

Let $\pi$ be a representation of $A$, the unital C*-algebra
generated by $S$ in $B(H)$.  The question we seek to answer in this section is, roughly, the extent
to which the ucp extensions of $\pi|_{S}$ to $A$ are determined by the
irreducible representations of $A$. More specifically, if every
irreducible representation of $A$ is a boundary representation for
$S$, must $\pi|_{S}$ extend uniquely to $A$?  We might start by
considering the identity representation of a Type I C*-algebra. So suppose $\psi$ is a ucp
map on $A$ (with deliberately ambiguous codomain) that satisfies
$\psi(s)=s$ for all $s\in S$, and assume every irreducible
representation of $A$ is a boundary representation for $S$. We know from Arveson's extension
theorem that there exists a ucp map from $A$ to $B(H)$ extending
$\psi$, but when $A$ is Type I, there is also a ucp map to $A''$
extending $\psi$ because $A''$ is injective.  If one could show that the ucp map with codomain
$B(H)$ is a representation (for every faithful representation)  it would effectively solve Arveson's
hyperrigidity conjecture (see \cite{Arveson08a} and 
Section~\ref{sect4}).  In this section, we show that the ucp
map with codomain $A''$ must be a representation (Corollary~\ref{maincor}). We will also give
some information in the more general case in Section~\ref{sect4}.

We recall some basic facts about direct integrals of Hilbert spaces
and decomposable operators.  Let $(X,\mu)$ be a standard Borel measure
space and let $H_x$ be a separable Hilbert space for each $x\in X$. A
\emph{measurable field of Hilbert spaces} is a vector subspace
$\mc{V}$ of $\Pi_{x\in X}H_x$ satisfying some natural measurability
conditions (see \cite{Blackadar06}), and whose elements we refer to as
\emph{measurable vector fields}.  One can obtain a Hilbert space $H$
from $\mc{V}$ by considering the set of measurable vector fields $\xi$ such that 
\begin{align*}  
\|\xi\| := \left ( \int_X \|\xi(x)\|^2\,d\mu(x)\right)^{1/2} < \infty.
\end{align*}
Identifying vector fields which agree almost everywhere, and defining
\begin{align*}
\langle \xi,\eta\rangle &:= \int_X\langle \xi(x),\eta(x)\rangle_{x}\,d\mu(x),
\end{align*}
one can show that $H$ is a Hilbert space.  We write
$H=\int_{X}^{\oplus} H_x\,d\mu(x)$. For $a_x\in B(H_x)$, $(a_x)$ is a
\emph{measurable field of bounded operators} if $(a_x\xi(x))$ is a
measurable vector field for each measurable vector field $\xi$.  When the family $(a_x)$ is uniformly
bounded in norm, it defines an operator $a$ on $B(H)$. This operator
is said to be \emph{decomposable} and  is written as
$a=\int_{X}^{\oplus}a_x\,d\mu(x)$. Its norm is
$\tmop{ess-sup}\,\|a_x\|$.

The main technical result is the following.

\begin{theorem}\label{thm4}
Let $S$ be a separable operator system in $B(H)$ generating a nuclear 
C*-algebra $A$. Suppose every
factor representation of $A$ has the UEP relative to $S$. Let $\rho$
be a faithful representation of $A$ on $B(K)$ and let $\gamma:\rho(A)\to
B(K)$ be a ucp map extending $\tmop{id}_{\rho(S)}$. Then for every conditional
expectation $E:B(K)\to \rho(A)''$, we have $E\gamma\rho(a)=\rho(a)$
for all $a\in A$.

\end{theorem}
\begin{proof}
We first prove the result when the Hilbert space $K$ is separable. Let
$E:B(K)\to \rho(A)''$ be a conditional expectation. Let $\gamma:B(K)\to
B(K)$ be a ucp map such that $\gamma\rho(s)=\rho(s)$ for all $s\in
S$. We will show that $E\gamma\rho=\rho$ for any conditional
expectation $E$, under
the assumption that every factor representation of $\rho(A)$ has the
UEP relative to $\rho(S)$.

Consider the commutative von Neumann algebra
$M:=\mc{Z}(\rho(A)'')$. Since it acts on a separable Hilbert space,
there is a weak* dense unital commutative separable C*-subalgebra $M_0$ of $M$. Let $X$ be the spectrum of $M_0$, so that $M_0\cong C(X)$.  There is a probability measure $\mu$ on $X$ such that $M\cong L^{\infty}(X,\mu)$. This gives us a disintegration $K=\int_X^{\oplus}K_x\,d\mu$, and the identity representation of $\rho(A)''$ may be decomposed as
\begin{align*}
b &= \int_{X}^{\oplus}\pi_x(b)\,d\mu(x),
\end{align*}
for all $b\in \rho(A)''$ (\cite[4.12]{Pedersen79}). After discarding a set of measure zero from $X$, the resulting set (which we still call $X$) has the property that each $\pi_x|_{\rho(A)}$ is a factor representation of $\rho(A)$. Since $E\gamma\rho(A)$ is contained in $\rho(A)''$, we may write
\begin{align*}
E\gamma\rho(a) &= \int_{X}^{\oplus}\pi_x(E\gamma\rho(a))\,d\mu(x),
\end{align*}
for all $a\in A$. Note that $\pi_x\rho$
is a factor representation of $A$ for all $x\in X$. Now
$\gamma\rho|_{S}=\rho|_{S}$ means that $\pi_x
E\gamma\rho|_{S}=\pi_x\rho|_{S}$ for a.e. $x\in X$. Because we
assumed that every factor representation of $\rho(A)$ has the UEP, we
conclude that $\pi_x
E\gamma\rho=\pi_x\rho$ for a.e. $x\in X$; from this it follows
that $E\gamma\rho=\rho$. 

Now assume that $K$ is not necessarily separable. Because $A$ is
separable, the representation
$\rho$ is unitarily equivalent to $\oplus \rho_i$, where each $\rho_i$
is a representation acting on a separable Hibert space $K_i$.  So it
suffices to show the claim for $\rho:=\oplus \rho_i$. Fix a
faithful separable representation $\sigma$ of $A$ on $B(L)$ with
conditional expectation $F:B(L)\to \sigma(A)''$.  Then $\rho_i\oplus
\sigma$ is a faithful separable representation of $A$ for each $i$.  Let $P_i$ be the projection of
$K$ onto $K_i$; note that $P_i\in \rho(A)'$ for each $i$. Consider the conditional
expectation $\tmop{Ad}\,P_i\circ E\oplus
 F:B(K_i\oplus L)\to (\rho_i\oplus \sigma)(A)''$. Using the result for
 separable representations above, we have $(\tmop{Ad}\,P_i\circ E\oplus
 F)(\gamma\rho\oplus \sigma)(a)=(\rho_i\oplus \sigma)(a)$ for all
 $a\in A$. Thus $(E\oplus F)(\gamma\rho\oplus \sigma)=\rho\oplus
 \sigma$, and we conclude $E\gamma\rho=\rho$.
\end{proof}

\begin{cor}\label{morecor}
Let $S$ be a separable operator system generating a Type I
C*-algebra $A$. If every irreducible representation of $A$  is a
boundary representation for $S$, then for any representation $\pi$ of
$A$ on $B(K)$ and any ucp map $\psi:\pi(A)\to B(K)$ extending
$\tmop{id}_{\pi(S)}$ and any conditional expectation $E:B(K)\to
\pi(A)''$, $E\psi\pi=\pi$. 
\end{cor}
\begin{proof}
Fix a faithful representation $\rho$ of $A$ and a conditional
expectation $F:B(K)\to \rho(A)''$.  We can apply
Theorem~\ref{thm4} to the faithful representation $\rho\oplus \pi$
using the conditional expectation $F\oplus E$;  we
conclude that $(F\oplus E)(\rho\oplus \psi\pi)(a)=(\rho\oplus
\pi)(a)$ for all $a\in A$, and so $E\psi\pi=\pi$. 
\end{proof}

\begin{cor}\label{maincor}
Let $S$ be a separable operator system generating a Type I
C*-algebra $A$.  If every irreducible representation of $A$  is a
boundary representation for $S$, then for any ucp map $\psi:A\to A''$ such that
$\psi(s)=s$, we have $\psi(a)=a$.
\end{cor}

\begin{proof}
We apply Theorem~\ref{thm4}, taking $\rho$ to be the identity representation.  When $A$ is Type I,
every factor representation is a multiple of an
irreducible representation. If every irreducible representation is a
boundary representation, direct sums of irreducibles will have the
UEP (\cite{DM05}).  So the hypotheses of the previous theorem are satisfied.  Because
$\psi(A)\subseteq A''$, we have $E\psi=\psi$, and so $\psi(a)=a$ for
all $a\in A$ follows immediately.
\end{proof}

\section{Hyperrigidity revisited}\label{sect4}

Arveson formulated a noncommutative version of a Korovkin set as follows.  Let $S$
be a concrete operator system and let $A$ be its generated
C*-algebra.  We call a representation $\pi:A\to B(K)$
\emph{hyperrigid} if whenever $\{\phi_n\}\subset \ucp{\pi(A)}{B(K)}$
satisfies $\|\phi_n\pi(s)-\pi(s)\|\to 0$ for all $s\in S$, then
$\|\phi_n\pi(a)-\pi(a)\|\to 0$ for all $a\in A$. A hyperrigid
representation must have the UEP; the converse of this is probably
false. We say that an operator system $S$ is
hyperrigid if every 
faithful representation of $A$ is hyperrigid. Arveson introduced this
notion in \cite{Arveson08a} while exploring Korovkin-type theorems for
C*-algebras.  We argued in Section~\ref{sect2} that it is more natural to require that
$\{\phi_n\pi\}$ (or a subsequence) converges to $\pi$ in the bounded-weak
or bounded-strong(*) topology on $\ucp{A}{B(K)}$, as this leads to a
characterization of representations having the UEP.

In \cite{Arveson08a}, Arveson proves that when $S$ is separable,  $S$ is hyperrigid if and only if every representation has the UEP relative to $S$. It follows that when $S$ is hyperrigid, every irreducible representation is a boundary representation.  The converse of this statement is the ``hyperrigidity conjecture''.  

\begin{conj}\cite[4.3]{Arveson08a}\label{conj7}
Let $S$ be a separable operator system in $B(H)$. If every irreducible representation of $A:=C^{\ast}(S)$ is a boundary representation for $S$, then $S$ is hyperrigid.
\end{conj}

The truth of the conjecture remains unknown even for commutative
C*-algebras.  Dritschel and McCullough \cite{DM05} showed it is true when the
C*-algebra generated by $S$ has a countable spectrum.  Put in
another way, the conjecture is true when $A''$ is purely atomic.  Arveson
later obtained a partial result for commutative C*-algebras:
he showed that if every irreducible representation of $C(X)$ has the
UEP relative to a function space generating $C(X)$, then there is a
``local'' unique extension property for arbitrary representations.
To prove the conjecture, it suffices to show the following (using the
notation from Theorem~\ref{thm4}): if every irreducible representation
of $A$ is a boundary representation for $S$, then for every
faithful representation $\rho:A\to B(K)$ and every ucp map
$\gamma:A\to B(K)$ extending $\tmop{id}_{\rho(S)}$,
$C^{\ast}(\gamma\rho(A))=\rho(A)$.  From now on, we will write $B$ for
$C^{\ast}(\gamma\rho(A))$. 

There are more consequences one may derive from Theorem~\ref{thm4}
that reveal how the C*-algebras $A$ and $B$ are related.
Recall that a \emph{boundary ideal} is an ideal $J$ of $A$ such that the
quotient map $q_J$ is completely isometric on $S$. A boundary ideal
that contains all other boundary ideals is the \emph{Shilov ideal} $\mathfrak{S}$;
such an ideal exists by \cite{Hamana79}. There is a unique
``smallest'' C*-algebra generated by $S$, called the
\emph{C*-envelope} of $S$, denoted $C_e^{\ast}(S)$,
and it is $\ast$-isomorphic to $A/\mathfrak{S}$. Alternatively, the
C*-envelope of $S$ is the C*-algebra generated by $S$
in its injective envelope $I(S)$ with the Choi-Effros multiplication
(see below).  

For a ucp map $\psi:A_1\to A_2$ where $A_1,A_2$ are unital
C*-algebras, denote by $\tmop{Mult}(\psi)$ the
\emph{multiplicative domain} of $\psi$; that is, the set 
\begin{align*}
\{x \in A_1: \psi(x^{\ast}x)=\psi(x)^{\ast}\psi(x) \text{ and }\psi(xx^{\ast})=\psi(x)\psi(x)^{\ast}\}.
\end{align*}
$\tmop{Mult}(\psi)$ is a unital C*-subalgebra of $A_1$ and
$\psi$ restricted to this set is a $\ast$-homomorphism. 

\begin{theorem}\label{othercor}
Let $S$, $A$, $\rho$, $E$, and $\gamma$ be as in Theorem~\ref{thm4}. Suppose every factor
representation of $A$ has the UEP relative to $S$. Let $B$ be the unital 
C*-algebra generated by the operator system $\gamma\rho(A)$ in $B(K)$.
\begin{enumerate}
\item $A\subseteq B\subseteq \tmop{Mult}(E)$ and $E|_{B}$ is a
  surjective idempotent $\ast$-homomorphism of $B$ onto $\rho(A)$;
\item $\rho(A)$ is $\ast$-isomorphic via $E|_{B}$ to the C*-envelope of the
  operator system $\gamma\rho(A)$, and so $\ker E|_{B}$ is the Shilov
  ideal of $\gamma\rho(A)$ in $B$;
\item If $\sigma$ is a representation of $B$ which factors through
  $\ker E|_{B}$, then $\sigma|_{\gamma\rho(A)}$ has the UEP.
\end{enumerate}
\end{theorem}
  
\begin{proof}\leavevmode
\begin{enumerate}
\item First, the operator system $\rho(S)$ is contained in $\gamma\rho(A)$
  since we have assumed that $\gamma\rho(s)=\rho(s)$ for all $s\in
  S$. So $\rho(A)$, the unital C*-algebra generated by
  $\rho(S)$, must be contained in $B$. Second, by Theorem~\ref{thm4},
  $E\gamma\rho=\rho$. Applying the Kadison-Schwarz inequality several
  times, we have
\begin{align*}
\rho(a^{\ast}a)=E\gamma\rho(a)^{\ast}E\gamma\rho(a) &\leq
E(\gamma\rho(a)^{\ast}\gamma\rho(a))\\ &\leq E\gamma\rho(a^{\ast}a)\\ &=\rho(a^{\ast}a),
\end{align*}
for all $a\in A$. So $E\gamma\rho(a)^{\ast}E\gamma\rho(a)=
E(\gamma\rho(a)^{\ast}\gamma\rho(a))$ for all $a\in A$, which implies
$\gamma\rho(A)\subseteq \tmop{Mult}(E)$. Because $\tmop{Mult}(E)$ is a
unital C*-algebra containing $\gamma\rho(A)$, it must also
contain $B$. Finally, the last statement is obvious: $E$ is a $\ast$-homomorphism when
restricted to $\tmop{Mult}(E)$, and its image is $\rho(A)$.

\item We may consider $\gamma\rho(A)$ as a subset of its injective envelope
in $B(K)$; that is, $\gamma\rho(A)\subseteq I(\gamma\rho(A))\subseteq
B(K)$. Let $F:B(K)\to B(K)$ be a completely positive norm 1 projection onto
$I(\gamma\rho(A))$. By (\cite{ChoiEffros77}), the map $F$ induces a multiplication on
$I(\gamma\rho(A))$: $x\cdot y := F(xy)$ for all $x,y\in
I(\gamma\rho(A))$. The C*-envelope of $\gamma\rho(A)$ is the
C*-subalgebra of $I(\gamma\rho(A))$ generated by
$\gamma\rho(A)$ in this multiplication.  Also, if we define $\gamma\rho(a)\circ
\gamma\rho(b) := \gamma\rho(ab)$ for all $a,b\in A$, it follows from
the fact that $\rho$ is faithful and $E\gamma\rho=\rho$ that this is also a
multiplication on $\gamma\rho(A)$ making it into a
C*-algebra. A multiplication on $\gamma\rho(A)$ making it
into a C*-algebra is essentially unique: there exists a
complete order isomorphism
$\alpha:
(\gamma\rho(A),\circ )\to C_{e}^{\ast}(\gamma\rho(A))$ such that $\alpha(\gamma\rho(a)\circ
\gamma\rho(b))=\gamma\rho(a)\cdot \gamma\rho(b)$ for all $a,b\in A$.  Thus we have
\begin{align*}
\alpha(\gamma\rho(ab)) &= \alpha(\gamma\rho(a)\circ
  \gamma\rho(b))\\  &=\gamma\rho(a)\cdot
  \gamma\rho(b)\\ &=F(\gamma\rho(a)\gamma\rho(b)),\qquad \forall
                    a,b\in A.
\end{align*}
When $b=1$, we get
$\alpha(\gamma\rho(a))=F\gamma\rho(a)=\gamma\rho(a)$. Thus
\begin{align*}
\gamma\rho(a)\circ \gamma\rho(b)&=\gamma\rho(ab) \\
                                &=\alpha(\gamma\rho(ab)) \\ &=
                                                              F(\gamma\rho(a)\gamma\rho(b))\\
  &=\gamma\rho(a)\cdot
\gamma\rho(b),\qquad \forall a,b\in A.
\end{align*}
In other words $\circ$ and $\cdot$ are the same multiplication. We
conclude that 
$(\gamma\rho(A),\circ)$ is the C*-envelope of $\gamma\rho(A)$. But
$E$ is a $\ast$-isomorphism of $(\gamma\rho(A),\circ)$ onto $\rho(A)$ and
the claim is proved.
\item Let $\psi$ be a ucp extension of $\sigma|_{\gamma\rho(A)}$ to
$B$ and let $J$ be $\ker E|_{B}$. Note that, by hypothesis, $\sigma=\pi q_J$, where $\pi$ is an isometric
$\ast$-homomorphism of $B/J$ onto $\sigma(B)$. Since $B/J\cong
\rho(A)$, without loss of generality, we may write $\sigma=
\pi E$. We have assumed that
$\psi(\gamma\rho(a))=\sigma(\gamma\rho(a))$, which in turn is
$\pi E(\gamma\rho(a))$, for all $a\in A$.  Because $E\gamma\rho=\rho$,
we have $\psi(\gamma\rho(a))=\pi\rho(a)$ for all $a\in A$. It is now easy
to see, with the usual multiplicative domain argument, that $\psi$ is
multiplicative on $B$: for all $a\in A$,
\begin{align*}
\pi\rho(a^{\ast}a)=\psi(\gamma\rho(a))^{\ast}\psi(\gamma\rho(a))&\leq
\psi(\gamma\rho(a)^{\ast}\gamma\rho(a)) \\
&\leq \psi\gamma\rho(a^{\ast}a) \\
&=\pi\rho(a^{\ast}a).
\end{align*}
This implies
$\psi(\gamma\rho(a))^{\ast}\psi(\gamma\rho(a))=\psi(\gamma\rho(a)^{\ast}\gamma\rho(a))$
for all $a\in A$ 
and it follows that $\psi$ is multiplicative.

\end{enumerate}
\end{proof}

\bibliographystyle{amsalpha}
\bibliography{docs}

\end{document}